\documentclass[final,onefignum,onetabnum]{siamart171218}

\usepackage{inputenc}
\usepackage{amsmath,amssymb,amsfonts,color,bm}
\usepackage{graphicx}
\usepackage{epstopdf}
\usepackage{caption}
\usepackage{latexsym}
\usepackage{multirow}
\usepackage{subcaption}
\usepackage{float}
\usepackage{verbatim}
\usepackage{enumitem}
\usepackage{algorithmic}
\usepackage{amsopn}

\def\R{{\mathbb R}}

\def\N{{\mathbb N}}

\newcommand{\ts}{\textstyle}
\newcommand{\ds}{\displaystyle}
\newcommand{\bs}{\boldsymbol}

\ifpdf
  \DeclareGraphicsExtensions{.eps,.pdf,.png,.jpg}
\else
  \DeclareGraphicsExtensions{.eps}
\fi

\newsiamremark{remark}{Remark}
\newsiamthm{assumption}{Assumption}
\newsiamremark{hypothesis}{Hypothesis}
\crefname{hypothesis}{Hypothesis}{Hypotheses}
\newsiamthm{claim}{Claim}

\headers{Shape-driven interpolation with discontinuous kernels}{S. De Marchi, W. Erb, F. Marchetti, E. Perracchione, M. Rossini}

\title{Shape-driven interpolation with discontinuous kernels: error analysis, edge extraction and applications in MPI}

\author{S. De Marchi\thanks{Dipartimento di Matematica \lq\lq Tullio Levi-Civita\rq\rq, Universit\`a di Padova, \email{demarchi@math.unipd.it}}
\and W. Erb\thanks{Dipartimento di Matematica \lq\lq Tullio Levi-Civita\rq\rq, Universit\`a di Padova, \email{erb@math.unipd.it}}
\and F. Marchetti\thanks{Dipartimento di Salute della Donna e del Bambino, Universit\`a di Padova, \email{francesco.marchetti.1@phd.unipd.it}}
\and E. Perracchione\thanks{Dipartimento di Matematica \lq\lq Tullio Levi-Civita\rq\rq, Universit\`a di Padova, \email{emma.perracchione@math.unipd.it}}
\and M. Rossini\thanks{Dipartimento di Matematica e Applicazioni, Universit\`a di Milano - Bicocca, \email{milvia.rossini@unimib.it }}
}

\ifpdf
\hypersetup{
  pdftitle={Shape-driven interpolation with discontinuous kernels},
  pdfauthor={S. De Marchi, W. Erb, F. Marchetti, E. Perracchione, M. Rossini}
}
\fi

\usepackage[twoside=false]{geometry}

\begin{document}

\maketitle

\begin{abstract}
Accurate interpolation and approximation techniques for functions with discontinuities are key tools in many applications as, for instance, medical imaging. In this paper, we study an RBF type method for scattered data interpolation that incorporates discontinuities via a variable scaling function. For the construction of the discontinuous basis of kernel functions, information on the edges of the interpolated function is necessary. We characterize the native space spanned by these kernel functions and study error bounds in terms of the fill distance of the node set. To extract the location of the discontinuities, we use a segmentation method based on a classification algorithm from machine learning. The conducted numerical experiments confirm the theoretically derived convergence rates in case that the discontinuities are a priori known. Further, an application to interpolation in magnetic particle imaging shows that the presented method is very promising.
\end{abstract}



\section{Introduction} \label{sec1}

Data interpolation is an essential tool in medical imaging. It is required for geometric alignment, registration of images, to enhance the quality on display devices, or to reconstruct the image from a compressed amount of data
\cite{Bushberg,Lehmann,Thevenaz}. Interpolation techniques are needed in the generation of images as well as in post-processing steps. In medical inverse problems as computerized tomography (CT) and magnetic resonance imaging (MRI), interpolation is used in the reconstruction process in order to fit the discrete Radon data into the back projection step. In single-photon emission computed tomography (SPECT) regridding the projection data improves the reconstruction quality while reducing acquisition times \cite{Takaki}. In Magnetic Particle Imaging (MPI), the number of calibration measurements can be reduced by interpolation methods \cite{KEASKB2016}.

In a general interpolation framework, we are given a finite number of data values sampled from an unknown function $f$ on a node set $\mathcal{X}\in\Omega,$ $\Omega\subseteq \R^d$. The goal of every interpolation scheme is to recover, in a faithful way, the function $f$ on the entire domain $\Omega$ or on a set of evaluation points. The choice of the interpolation model plays a crucial role for the quality of the reconstruction. If the function $f$ belongs to the interpolation space itself, $f$ can be recovered exactly. On the other hand, if the basis of the interpolation space does not reflect the properties of $f$, artifacts will usually appear in the reconstruction. In two-dimensional images such artifacts occur for instance if the function $f$ describing the image has sharp edges, i.e. discontinuities across curves in the domain $\Omega$. In this case, smooth interpolants get highly oscillatory near the discontinuity points.

This is a typical example of the so-called Gibbs phenomenon. This phenomenon was originally formulated in terms of overshoots that arise when univariate functions with jump discontinuities are approximated by truncated Fourier expansions, see \cite{Zygmund}. Similar artifacts arise also in higher dimensional Fourier expansions and when interpolation operators are used. In medical imaging like CT and MRI, such effects are also known as ringing or truncation artifacts \cite{Czervionke}.

The Gibbs phenomenon is also a well-known issue for other basis systems like wavelets or splines, see \cite{Jerri} for a general overview.  Further, it appears also in the context of radial basis function (RBF) interpolation \cite{Fornberg_gibbs}. The effects of the phenomenon can usually be softened by applying additional smoothing filters to the interpolant.
For RBF methods, one can for instance use linear RBFs in regions around discontinuities \cite{Jung}. Furthermore, post-processing techniques, such as Gegenbauer reconstruction procedure \cite{Gottlieb} or digital total variation \cite{sarra_post}, are available.

\subsection*{Main contributions}
In this work, we provide a shape-driven method to interpolate \emph{scattered} data sampled from  discontinuous functions. Novel in our approach is that the interpolation space is modelled according to edges (known or estimated) of the function. In order to do so we consider variably scaled kernels (VSKs) \cite{Bozzini1,rossini} and their discontinuous extension \cite{demarchi18}. Starting from a classical kernel $K$, we define a basis that reflects discontinuities in the data. These basis functions, referred to as variably scaled discontinuous kernels (VSDKs), strictly depend on the given data. In this way, they intrinsically provide an effective tool capable to interpolate functions with given discontinuities in a faithful way and to avoid overshoots near the edges.

If the edges of the function $f$ are explicitly known, we show that the proposed interpolation model outperforms the classical RBF interpolation and avoids Gibbs artifacts.
From the theoretical point of view, we provide two main results. If the kernel $K$ is algebraically decaying in the Fourier domain, we characterize the native space of the VSDK as a piecewise Sobolev space. This description allows us to derive in a second step convergence rates of the discontinuous interpolation scheme in terms of a fill distance for the node set in the domain $\Omega$. The VSDK convergence rates are significantly better than the convergence rates for standard RBF interpolation. Numerical experiments confirm the theoretical results and point out that even better rates are possible if the kernel $K$ involved in the definition of the VSDK is analytic.

In applied problems, as medical imaging, the edges of $f$ are usually not a priori known. In this case, we need reliable edge detection or image segmentation algorithms that estimate edges (position of jumps) from the given data. For this reason, we encode in the interpolation method an additional segmentation process based on a classification algorithm that provides the edges via a kernel machine. As labels for the classification algorithm we can use thresholds based on function values or on RBF coefficients \cite{romani}. The main advantage of this type of edge extraction process is that it works directly for scattered data, in contrast to other edge detection schemes such as Canny or Sobel detectors that usually require an underlying grid (cf. \cite{Canny,Sharifi}).

\subsection*{Outline} In Section \ref{sec2}, we recall the basic notions on kernel based interpolation. 
Then, in Section \ref{sec2b} we present the theoretical findings on the characterization of the VSDK native spaces (if the discontinuities are known) and Sobolev-type error estimates of the corresponding interpolation scheme. In
Section \ref{sec3}, numerical experiments confirm the theoretical convergence rates and reveal that if the edges are known, the VSDK interpolant outperforms the classical RBF interpolation. Beside these experiments, we show how the interpolant behaves with respect to perturbations of the scaling function that models the discontinuities. We review image segmentation via classification and machine learning tools in Section \ref{sec4} and summarize our new approach. In Section \ref{sec5}, the novel VSDK interpolation technique incorporating the segmentation algorithm is applied to Magnetic Particle Imaging. Conclusions are drawn in Section \ref{sec6}.

\section{Preliminaries on kernel based interpolation and variably scaled kernels} \label{sec2}

Kernel based methods are powerful tools for scattered data interpolation. In the following, we give a brief overview over the basic terminology. For the theoretical background and more details on kernel methods, we refer the reader to \cite{buhmann03,Fasshauer15,Wendland05}.

\subsection{Kernel based interpolation}
For a given set of scattered nodes ${\cal X} = \{\boldsymbol{x}_1, \ldots, \boldsymbol{x}_N\} \subseteq \Omega$, $\Omega \subseteq \R^d$, and values $f_i \in \R$, $i \in \{1,\ldots, N\}$, we want to find a function $P_f: \Omega \to \mathbb{R}$ that satisfies the interpolation conditions
\begin{equation}
P_f\left( \boldsymbol{x}_i\right) = f_i, \quad i\in \{1,\ldots, N\}.
\label{eq0}
\end{equation}	

We express the interpolant $P_f$ in terms of a kernel
$K: \R^d \times \R^d \to \R$, i.e.,
 \begin{equation}
 \label{eq1}
 P_f\left( \boldsymbol{x}\right)= \sum_{k=1}^{N} c_k K \left( \boldsymbol{x} , \boldsymbol{x}_k  \right), \quad \boldsymbol{x} \in \Omega.
 \end{equation}
 
If the kernel $K$ is symmetric and strictly positive definite, the matrix $A = (A_{ij})$ with the entries 
$ A_{ij}= K \left( \boldsymbol{x}_i , \boldsymbol{x}_j  \right)$, $1\leq i,j\leq N$, is positive definite for all possible sets of nodes. In this case, the coefficients $c_k$ are uniquely determined by the interpolation conditions in \eqref{eq0} and can be obtained by solving the
linear system
$ A \boldsymbol{c} = \boldsymbol{f},$ where $  \boldsymbol{c}= \left(c_1, \ldots,
 c_N\right)^{\intercal}$, and $  \boldsymbol{f} =\left(f_1, \ldots , f_N\right)^{\intercal}$.

Moreover, there exists a so-called native space for the kernel $K$, that is a Hilbert space ${\cal N}_{K}(\Omega)$ with inner product $(\cdot, \cdot)_{{\cal	N}_{K}(\Omega)}$ in which the kernel $K$ is reproducing, i.e., for any $f \in {\cal	N}_{K}(\Omega)$ we have the identity
\[ f(x) = (f,K(\cdot,x))_{{\cal	N}_{K}(\Omega)}, \quad x \in \Omega.\]
Following \cite{Wendland05}, we introduce the native space by defining the space
  \begin{equation*}
  H_K(\Omega) =  \mathrm{span} \left\{ K(\cdot,\boldsymbol{y}), \hskip 0.2cm \boldsymbol{y} \in \Omega \right\}
  \end{equation*}
equipped with the bilinear form 
\begin{equation} \label{eq:innerproduct}
\left(f,g\right)_{H_K(\Omega)} = \sum_{i=1}^N \sum_{j=1}^M a_i b_j K(\bs{x}_i,\bs{y}_j),
\end{equation}
where  $f,g \in H_K(\Omega)$ with $f(\bs{x}) = \sum_{i=1}^N a_i K(\bs{x},\bs{x}_i)$ and $g(\bs{x}) = \sum_{j=1}^M b_j K(\bs{x},\bs{y}_j).$
The space $\left(f,g\right)_{H_K(\Omega)}$ equipped with  $\left(f,g\right)_{H_K(\Omega)}$ is an inner product space with reproducing kernel $K$ (see \cite[Theorem 10.7]{Wendland05}). The native space ${\cal	N}_{K}(\Omega)$ of the kernel $K$ is then defined as the completion of $H_K(\Omega)$ with respect to the norm $||\cdot||_{H_K(\Omega)} = \sqrt{(\cdot, \cdot)_{H_K(\Omega)}}$. In particular for all $f \in H_K(\Omega)$ we have $||f||_{{\cal	N}_{K}(\Omega)}= ||f||_{H_K(\Omega)}$.

\subsection{Variably scaled kernels} \label{subsec:VSK}

Variaby scaled kernels (VSKs) were introduced in \cite{Bozzini1}. They depend on a \emph{scaling} function $\psi\;:\;\R^d\rightarrow \R.$ 
\begin{definition}\label{def_vsk}
	Let $K: \R^{d+1} \times \R^{d+1} \to \R$ be
	a continuous strictly positive definite kernel. Given a scaling function $\psi\;:\;\R^d\rightarrow \R,$
	 a variably scaled kernel $K_{\psi}$ on $\R^d \times \R^d$ is defined as
	\begin{equation}\label{def_vsk_eq}
	K_{\psi}(\bs{{x}},\bs{{y}}):= K((\bs{{x}},\psi(\bs{x})),(\bs{y},\psi(\bs{y})),
	\end{equation}
	for $\bs{x},\bs{y}\in\R^d$.
\end{definition}
The so given VSK $K_{\psi}$ is strictly positive definite on $\R^d$. 
Suitable choices of the scaling function $\psi$ allow to improve stability and recovery quality of the 
kernel based interpolation, as well as to preserve shape properties of the original function, 
see e.g. the examples in \cite{Bozzini1}, \cite{demarchi18} and \cite{rossini}.

In this paper we consider  radial kernels
$K: \R^d \times \R^d \to \R$, i.e.,
\begin{equation} \label{def:rbfkernel}
K(\boldsymbol{x},\boldsymbol{y}) = \phi( ||\boldsymbol{x}-\boldsymbol{y}||_2), \quad \boldsymbol{x},\boldsymbol{y}\in \Omega,
\end{equation}
with a continuous {scalar} function $ \phi: [0, \infty) \to \mathbb{R}$. The function $\phi$ is called   radial basis function (RBF). In this case a VSK has the form
\begin{equation}\label{eqvskr}
K_{\psi}(\bs{{x}},\bs{{y}}) = \phi \left( \sqrt{||\boldsymbol{x}-\boldsymbol{y}||_2^2 + |\psi(\bs{x})-\psi(\bs{y})|^2}\right).
\end{equation}

%The RBF interpolant $P_f$ takes the form
  
\subsection{Variably scaled kernels with discontinuities} \label{subsec:VSDK}
Our goal is to introduce interpolation spaces based on discontinuous basis functions on $\Omega$. For the definition of these spaces, we use a piecewise continuous scaling function $\psi$. The associated VSK $K_{\psi}(\bs{{x}},\bs{{y}})$ is then also only piecewise continuous and denoted as variably scaled discontinuous kernel (VSDK). 

We consider the following setting:
\begin{assumption} \label{ass:1}
We assume that:
\begin{itemize}
\item[(i)] The bounded set $\Omega\subset \R^d$ is
the union of $n$ pairwise disjoint sets $\Omega_i$, $i \in \{1, \ldots, n\}$.
\item[(ii)] The subsets $\Omega_i$ satisfy an interior cone condition and have a Lipschitz boundary.
\item[(iii)] Let $\Sigma = \{\alpha_1, \ldots, \alpha_n\}$, $\alpha_i \in \R$. The function $\psi: \Omega \to \Sigma$ is piecewise constant so that $\psi(\bs{x}) = \alpha_i$ for all $\bs{x} \in \Omega_i$. In particular, the discontinuities of $\psi$ appear only at the boundaries of the subsets $\Omega_i$. We assume that $\alpha_i \neq \alpha_j$ if $\Omega_i$ and $\Omega_j$ are neighboring sets.
\end{itemize}
 
\end{assumption}

The kernel function $K_{\psi}(\bs{{x}},\bs{{y}})$ based on a piecewise constant scaling function 
$\psi$ is well defined for all $\bs{{x}},\bs{{y}} \in \Omega$. 
If $\bs{{x}}$ and $\bs{{y}}$ are contained in the same subset $\Omega_i \subset \Omega$ then
$K_{\psi}(\bs{{x}},\bs{{y}}) = K(\bs{{x}},\bs{{y}})$.
We denote the graph of the function $\psi$ with respect to the domain $\Omega$ by
$$G_{\psi}(\Omega) = \{(\bs{{x}},\psi(\bs{x})) \ | \ \bs{{x}} \in \Omega\} \subset \Omega \times \Sigma. $$

To have a more compact notation for the elements of the graph $G_{\psi}(\Omega)$, we use the shortcuts $\bs{\tilde{x}} = (\bs{{x}},\psi(\bs{x}))$ and $\bs{\tilde{y}} = (\bs{{y}},\psi(\bs{y}))$.
In the same way as in \eqref{eq1}, we can define an interpolant for
the nodes $\tilde{\mathcal{X}} = \{ \bs{\tilde{x}}_1, \ldots, \bs{\tilde{x}}_N\}$ on the graph $G_{\psi}(\Omega)$. Using the kernel $K$ on $G_{\psi}(\Omega) \subset \Omega \times \Sigma \subset \R^{d+1}$, we obtain an interpolant of the form
\begin{equation}
\label{eq2}
P_f\left(\bs{\tilde{x}}\right)= \sum_{k=1}^{N} c_k K( \bs{\tilde{x}},\bs{\tilde{x}}_k).
\end{equation}
Based on this interpolant on $G_{\psi}(\Omega)$, we define the VSDK interpolant
$V_f$ on $\Omega$ as
\begin{equation}
 \label{eq3}
V_f(\bs{x}) = P_f\left( \tilde{\boldsymbol{x}}\right) = \sum_{k=1}^{N} c_k K_{\psi}(\bs{{x}},\bs{{x}}_k), \quad \boldsymbol{x} \in \Omega.
\end{equation}

The coefficients $c_1, \ldots, c_N$ of the VSDK interpolant $V_f(\bs{x})$ in \eqref{eq3} are obtained
by solving the linear system of equations
\begin{equation}
\begin{pmatrix}
K(\tilde{\bs{x}}_1 , \tilde{\bs{x}}_1) & \cdots & K(\tilde{\bs{x}}_1 , \tilde{\bs{x}}_N)\\
\vdots & & \vdots\\
K(\tilde{\bs{x}}_N , \tilde{\bs{x}}_1) & \cdots & K(\tilde{\bs{x}}_N , \tilde{\bs{x}}_N)\\
\end{pmatrix}
\begin{pmatrix}
c_1 \\ \vdots \\ c_N
\end{pmatrix}
=
\begin{pmatrix}
f_1  \\ \vdots \\ f_N
\end{pmatrix}.
\label{eq:coeffVSDK}
\end{equation}

In fact, the so obtained coefficients are precisely the coefficients for the interpolant \eqref{eq2} for the node points ${\cal \tilde{X}}$ on
the graph $G_{\psi}(\Omega)$. Since the kernel $K : G_{\psi}(\Omega) \times G_{\psi}(\Omega) \to \R$ is strictly positive definite, the system \eqref{eq:coeffVSDK} admits a unique solution.
For the kernel $K : G_{\psi}(\Omega) \times G_{\psi}(\Omega) \to \R$ and the discontinuous kernel $K_{\psi}: \Omega \times \Omega \to \R$ we can further
define the two inner product spaces
  \begin{align*}
  H_K(G_{\psi}(\Omega)) &=  \mathrm{span} \left\{ K(\cdot,\tilde{\boldsymbol{y}}), \hskip 0.2cm \tilde{\boldsymbol{y}} \in G_{\psi}(\Omega) \right\},\\
  H_{K_\psi}(\Omega) &=  \mathrm{span} \left\{ K_{\psi}(\cdot,\boldsymbol{y}), \hskip 0.2cm \boldsymbol{y} \in \Omega \right\},
  \end{align*}
with the inner products given as in \eqref{eq:innerproduct}. For both spaces, we can take the completion and obtain in this way the native spaces
 ${\cal	N}_{K}(G_{\psi}(\Omega))$ and ${\cal N}_{K_{\psi}}(\Omega)$, respectively. We have the following relation between the two native spaces.

  \begin{proposition} \label{prop-extensiontograph}
  	The native spaces ${\cal N}_{K}(G_{\psi}(\Omega))$ and ${\cal	N}_{K_{\psi}}(\Omega)$
  	are isometrically isomorphic.
  \end{proposition}

In the same way as in \cite[Theorem 2]{Bozzini1}, Proposition \ref{prop-extensiontograph} follows from the fact that the two inner product spaces $H_K(G_{\psi}(\Omega))$ and $H_{K_\psi}(\Omega)$ are isometric. Then, the same holds true for their respective completion.

\section{Approximation with discontinuous kernels} \label{sec2b}

\subsection{Characterization of the native space for VSDKs}
Based on the decomposition of the domain $\Omega$ described in Assumption \ref{ass:1} we define for
$s \geq 0$ and $1 \leq p \leq \infty$ the following spaces of piecewise smooth functions on $\Omega$:
\[\mathrm{WP}_p^s(\Omega) := \left\{ f: \Omega \to \R \ | \ f_{\Omega_i} \in \mathrm{W}_p^s(\Omega_i), \quad i \in \{1, \ldots, n\} \right\}. \]
Here, $f_{\Omega_i}$ denotes the restriction of $f$ to the subregion $\Omega_i$ and
$\mathrm{W}_p^s(\Omega_i)$ denote the standard Sobolev spaces on $\Omega_i$.
As norm on $\mathrm{WP}_p^s(\Omega)$ we set
\[\|f\|_{\mathrm{WP}_p^s(\Omega)}^p = \sum_{i = 1}^n \|f_{\Omega_i}\|_{\mathrm{W}_p^s(\Omega_i)}^p. \]
The piecewise Sobolev space $\mathrm{WP}_p^s(\Omega)$ and the corresponding norm strongly depend on the chosen decomposition of the domain $\Omega$. However, for any decomposition of $\Omega$ in 
Assumption \ref{ass:1}, the standard Sobolev space $\mathrm{W}_p^s(\Omega)$ is contained in $\mathrm{WP}_p^s(\Omega)$. In the following, we assume that the radial kernel $K$ defining the VSDK $K_{\psi} $ 
has a particular Fourier decay:
\begin{equation} \label{Fourierdecay}
 \widehat {\phi(\|\cdot\|)}(\bs{\omega}) \sim (1+\|\bs{\omega}\|_2^2)^{-s-\frac{1}{2}}, \quad s > \frac{d-1}{2}.
\end{equation}

In order to characterize the native space ${\cal N}_{K_{\psi}}(\Omega)$, we need some additional results regarding the continuity of trace and extension operators. For the reader's convenience, we list some relevant
results from the literature.

\begin{lemma} \label{lem:exttrace} We have the following relations for extension and trace operators in the
native spaces (a) and in the Sobolev spaces (b):
\begin{enumerate}
\item[(a)] {\textnormal{(\cite[Theorem 10.46 \& Theorem 10.47]{Wendland05} or \cite[Section 9]{Schaback1999})}} Every $f \in {\cal N}_{K}(G_{\psi}(\Omega))$ has a natural extension $E f \in {\cal N}_{K}(\R^{d+1})$. Further, $$\|Ef\|_{{\cal N}_{K}(\R^{d+1})} = \|f\|_{{\cal N}_{K}(G_{\psi}(\Omega))}.$$
For every $g \in {\cal N}_{K}(\R^{d+1})$, the trace $T_{G_{\psi}(\Omega)} g$ is contained in ${\cal N}_{K}(G_{\psi}(\Omega))$. Further, $$\| T_{G_{\psi}(\Omega)} g\|_{{\cal N}_{K}(G_{\psi}(\Omega))}
\leq \|g\|_{{\cal N}_{K}(\R^{d+1})}.$$
\item[b)] {\textnormal{(\cite[Theorem 7.39]{adams2003})}}  Let $s > \frac{1}{2}$. For every $g \in \mathrm{W}^s_2(\R^{d+1})$ the trace $T_{G_{\psi}(\Omega)} g$ is contained in $\mathrm{W}^{s-1/2}_2(G_{\psi}(\Omega))$ and the trace operator
$T_{G_{\psi}(\Omega)}: \mathrm{W}^s_2(\R^{d+1}) \to \mathrm{W}^{s-1/2}_2(G_{\psi}(\Omega))$
is bounded. \\
Further, there exists a bounded extension operator $E: \mathrm{W}^{s-1/2}_2(G_{\psi}(\Omega)) \to \mathrm{W}^s_2(\R^{d+1})$ such that
$ T_{G_{\psi}(\Omega)} E  f = f$ for all $f \in \mathrm{W}^{s-1/2}_2(G_{\psi}(\Omega))$.
\end{enumerate}
\end{lemma}
We are now ready to prove the following theorem.
\begin{theorem} \label{thm:char}
Let  Assumption \ref{ass:1} hold true, and assume that the continuous strictly positive definite kernel $K: \R^{d+1} \times \R^{d+1} \to \R$ based on the radial basis function $\phi$ satisfies the decay condition \eqref{Fourierdecay}.
Then, for the discontinuous kernel $K_{\psi}$, we have
\[{\cal N}_{K_{\psi}}(\Omega) = \mathrm{WP}_2^s(\Omega),\]
with the norms of the two Hilbert spaces being equivalent.
\end{theorem}

\begin{proof}
We consider the following forward and backward chain of Hilbert space operators:
\[ {\cal N}_{K_{\psi}}(\Omega) \underset{Q_1}{\overset{P_5}{\leftrightarrows}} {\cal N}_{K}(G_{\psi}(\Omega)) \underset{Q_2}{\overset{P_4}{\leftrightarrows}}
{\cal N}_{K}(\R^{d+1}) \underset{Q_3}{\overset{P_3}{\leftrightarrows}} {\mathrm{W}_2^{s+1/2}}(\R^{d+1})
\underset{Q_4}{\overset{P_2}{\leftrightarrows}} {\mathrm{W}_2^{s}}(G_{\psi}(\Omega))  \underset{Q_5}{\overset{P_1}{\leftrightarrows}} \mathrm{WP}_2^s(\Omega).\]
In the backward direction, the operators $P_1, \ldots, P_5$ are given as
\begin{align*}
 P_1: \quad & P_1 f = \tilde{f}, \quad \text{with} \quad \tilde{f}(\tilde{\bs{x}}) = f(\bs{x}), \quad  \text{for all $\bs{x} \in \Omega$}, \\[-0.5mm]
 P_2: \quad & P_2 f = E f, \quad \text{(Extension in the sense of Lemma \ref{lem:exttrace} (b))} \\[-0.5mm]
 P_3: \quad & P_3 f = f, \\[-0.5mm]
 P_4: \quad & P_4 f = T_{G_\psi(\Omega)} f, \quad \text{(Trace in the sense of Lemma \ref{lem:exttrace} (a))} \\[-1mm]
 P_5: \quad & P_5 \tilde{f} = f, \quad \text{with} \quad f(\bs{x}) = \tilde{f}(\tilde{\bs{x}}),
 \quad \text{for all $\bs{x} \in \Omega$}.
\end{align*}
In the forward chain, the operators $Q_1, \ldots, Q_5$ are similarly defined as
\begin{align*}
 Q_1: \quad & Q_1 f = \tilde{f}, \quad \text{with} \quad \tilde{f}(\tilde{\bs{x}}) = f(\bs{x}), \quad  \text{for all $\bs{x} \in \Omega$}, \\[-0.5mm]
 Q_2: \quad & Q_2 f = E f, \quad \text{(Extension in the sense of Lemma \ref{lem:exttrace} (a))} \\[-0.5mm]
 Q_3: \quad & Q_3 f = f, \\[-0.5mm]
 Q_4: \quad & Q_4 f = T_{G_\psi(\Omega)} f, \quad \text{(Trace in the sense of Lemma \ref{lem:exttrace} (b))} \\[-1mm]
 Q_5: \quad & Q_5 \tilde{f} = f, \quad \text{with} \quad f(\bs{x}) = \tilde{f}(\tilde{\bs{x}}),
 \quad \text{for all $\bs{x} \in \Omega$}.
\end{align*}
All these 10 operators are well defined and continuous: $P_5$ and $Q_1$ are isometries by Proposition \ref{prop-extensiontograph}. $P_2$, $Q_2$ as well as $P_4$, $Q_4$ are continuous by Lemma \ref{lem:exttrace}.
Since $K$ satisfies the condition \eqref{Fourierdecay}, the native space ${\cal N}_{K}(\R^{d+1})$ is equivalent to the Sobolev space $\mathrm{W}_2^s(\R^{d+1})$ (see \cite[Corollary 10.48]{Wendland05}).
Therefore also the identity mappings $P_3$ and $Q_3$ are continuous. Finally,
the Sobolev norm for a function $\tilde{f}$ on the graph $G_{\psi}(\Omega)$ is a reformulation of the norm of $f \in \mathrm{WP}_2^s(\Omega)$. Therefore also the operators $P_1$ and $Q_5$ are isometries.

We can conclude that the concatenations $P_5 P_4 P_3 P_2 P_1$ and $Q_5 Q_4 Q_3 Q_2 Q_1$ are continuous operators. Since $P_5 P_4 P_3 P_2 P_1$ is the inverse to $Q_5 Q_4 Q_3 Q_2 Q_1$, these
two operators therefore provide an isomorphism between the Hilbert spaces ${\cal N}_{K_{\psi}}(\Omega)$
and $\mathrm{WP}_2^s(\Omega)$.
\end{proof}

\subsection{Error estimates for VSDK interpolation}

We state first of all a well-known Sobolev sampling inequality for functions vanishing on the subsets
${\cal X} \cap \Omega_i$ which was developed in \cite{Narcovich05}.
For this we introduce the following regional fill distance $h_i$ on the subset $\Omega_i$:
\[ h_i = \sup_{\bs{x} \in \Omega_i} \inf_{\bs{x}_i \in {\cal X} \cap \Omega_i} \| \bs{x} - \bs{x}_i \|_2.\]

\begin{proposition}[Theorem 2.12 in \cite{Narcovich05} or Proposition 9 in \cite{fuselier}] \label{prop:lenin}
Let $s > 0$, as well as $1 \leq p < \infty$ and $1 \leq q \leq \infty$. Further, let $m\in\N_0$ such that
$\lfloor s \rfloor > m + d/p$ (for $p=1$ also equality is possible) and $u$ be a function
that vanishes on ${\cal X} \cap \Omega_i$. Then, there is a $h_0 > 0$ such that for
$h_i \leq h_0$ and for the subregions $\Omega_i$ satisfying Assumption \ref{ass:1} $(ii)$ we have the Sobolev inequality
\[ \|u\|_{\mathrm{W}_q^{m}(\Omega_i)} \leq C_i h_i^{s-m-d(1/p-1/q)_+} \|u\|_{\mathrm{W}_p^{s}(\Omega_i)}.\]
The constant $C_i > 0$ is independent of $h_i$.
\end{proposition}

The Sobolev sampling inequalities given in Proposition \ref{prop:lenin} allow us to extract the correct
power of the fill distance from the smoothness of the underlying error function.
Based on these inequalities, a similar analysis can be conducted also on manifolds, see
\cite{fuselier}. Further sampling inequalities that can be used as a substitute for Proposition \ref{prop:lenin} can, for instance, be found in \cite[Theorem 2.1.1]{rieger}.

We define now the global fill distance
\[ h = \max_{i \in \{1, \ldots, n\}} h_i, \]
and get as a consequence of the regional sampling inequalities in Proposition \ref{prop:lenin} the following Sobolev error estimate:

\begin{theorem} \label{thm:errorestimate}
Let Assumption \ref{ass:1} be satisfied. Further, let $s > 0$, $1 \leq q \leq \infty$ and $m\in\N_0$ such that $\lfloor s \rfloor > m + \frac{d}{2}$. Additionally, suppose that the RBF $\phi$ satisfies the Fourier decay \eqref{Fourierdecay}. Then, for $f \in \mathrm{WP}_2^s(\Omega)$, we obtain for all $h \leq h_0$ the error estimate
\[ \|f - V_f\|_{\mathrm{WP}_q^{m}(\Omega)} \leq C h^{s-m-d(1/2-1/q)_+} \|f\|_{\mathrm{WP}_2^{s}(\Omega)}.\]
The constant $C > 0$ is independent of $h$.
\end{theorem}

\begin{proof}
By assumption, the function $f$ is an element of $\mathrm{WP}_2^s(\Omega)$. Further, the native space characterization in Theorem \ref{thm:char} guarantees that also the VSDK interpolant $V_f$ is an element of $\mathrm{WP}_2^s(\Omega)$.
Therefore, we can apply Proposition \ref{prop:lenin} with $p=2$ to every subset $\Omega_i$ and obtain
\[ \| f - V_f\|_{\mathrm{W}_q^{m}(\Omega_i)} \leq C_i h_i^{s-m-d(1/2-1/q)_+} \|f - V_f\|_{\mathrm{W}_2^{s}(\Omega_i)}, \quad i \in \{1, \ldots, n\},\]
with $h_i \leq h_0$. Now, using the definition of the piecewise Sobolev space $\mathrm{WP}_q^m(\Omega)$, we can synthesize these estimates to obtain
\begin{equation} \label{eq:errorintermediate}
 \| f - V_f\|_{\mathrm{WP}_q^{m}(\Omega)} \leq C h^{s-m-d(1/2-1/q)_+} \|f - V_f\|_{\mathrm{WP}_2^{s}(\Omega)},
\end{equation}
where $C = \max_{1 \leq i \leq n} C_i$ and $h = \max_{1 \leq i \leq n} h_i$. Since $\mathrm{WP}_2^{s}(\Omega)$ is equivalent to the native space ${\cal N}_{K_{\psi}}(\Omega)$ we can use the fact
that the interpolant $V_f$ is a projection into a subspace of ${\cal N}_{K_{\psi}}(\Omega)$.
This helps us to finalize our bound:
\[\|f - V_f\|_{\mathrm{WP}_2^{s}} \leq C' \|f - V_f\|_{{\cal N}_{K_{\psi}}(\Omega)}
\leq C' \|f\|_{{\cal N}_{K_{\psi}}(\Omega)} \leq C'' \|f\|_{\mathrm{WP}_2^{s}(\Omega)}, \]
with two constants $C'$, $C''$ describing the upper and lower bound for the equivalence of
the two Hilbert space norms.
\end{proof}

\begin{remark} The error estimates in Theorem \ref{thm:errorestimate} provide a theoretical explanation why VSDK interpolation  is superior to RBF interpolation in the spaces $\mathrm{WP}_2^s(\Omega)$. In these spaces the convergence of
the interpolant $V_f$ towards $f \in \mathrm{WP}_2^s(\Omega)$ depends only on the smoothness
$s$ of $f$ in the interior of the subsets $\Omega_i \subset \Omega$ and not on
the discontinuities at the boundaries of $\Omega_i$. If $s$ is sufficiently large, the corresponding fast convergence of the interpolation scheme prevents the emergence of Gibbs artifacts in the interpolant $V_f$. 
\end{remark}

\section{Numerical experiments} \label{sec3}

\subsection{Experimental setup} \label{sec31}
In our main application in magnetic particle imaging we will use samples along Lissajous trajectories as interpolation nodes. For this, we will introduce and use these node sets already for the numerical experiments in this section. As test images we consider the Shepp-Logan phantom and an additional simple geometric phantom. We give a brief description of this experimental setup.

\begin{figure}[htbp]
    \centering
	\includegraphics[width= 1\textwidth]{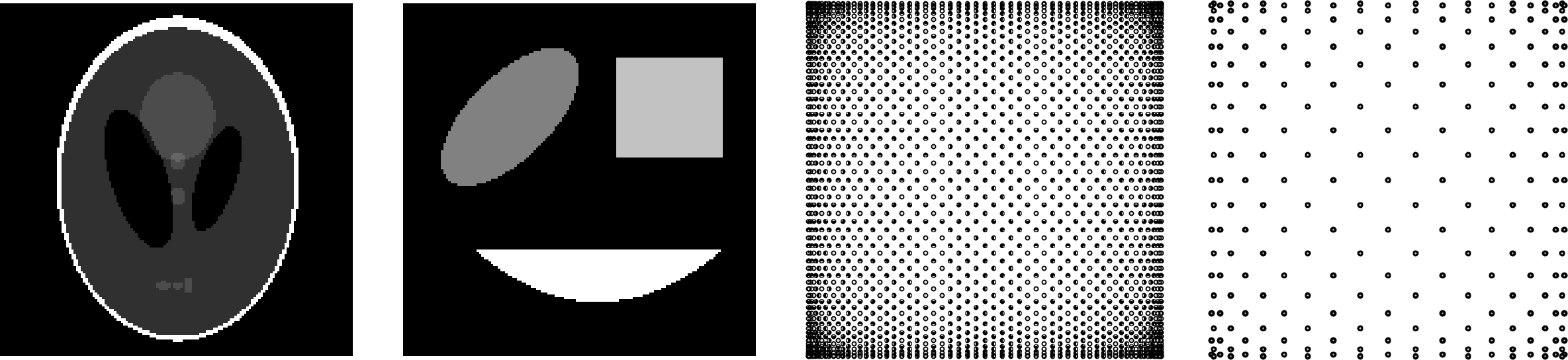}
	\caption{The Shepp-Logan phantom (left), a geometric phantom (middle, left), as well as the   Lissajous nodes $\bs{\mathrm{LS}}_{2}^{(32,33)}$ (middle, right) and
	$\bs{\mathrm{LS}}_{2}^{(10,11)}$ (right).}
	\label{fig:1}
\end{figure}

\subsubsection{Lissajous interpolation nodes}

For a vector $\bs{n}=(n_1,n_2)\in\N^2$ with relatively prime
frequencies $n_1$, $n_2$ and $\epsilon \in \{1,2\}$, the generating curves for the Lissajous nodes are given as
\begin{equation}\label{parallelismo}
\gamma_{\epsilon}^{(\bs{n})}(t)=
\left( \cos(n_2t), \, \cos \left( \ts n_1 t - \frac{\epsilon-1}{2n_2} \pi \right) \right).
\end{equation}
The Lissajous curve $\gamma_{\epsilon}^{(\bs{n})}$ is $2 \pi$-periodic and contained in the
square $[-1,1]^2$. If $\epsilon=1$, this curve is degenerate, i.e. it is traversed twice in one period. Further, $\gamma_{1}^{(\bs{n})}$ with $\bs{n} = (n,n+1)$, $n \in \N$, are
the generating curves of the Padua points \cite{bos,Erb2016}.
If $\epsilon=2$, then the curve is non-degenerate. If $n_1 + n_2$ is odd, the curve $\gamma_{2}^{(\bs{n})}(t)$ in
\eqref{parallelismo} can be further simplified in terms of two sine functions and gives a typical sampling trajectory encountered in magnetic particle imaging, see \cite{erb,EKDA2015,knopp,knopp_lissa}.
Using $\gamma_{\epsilon}^{(\bs{n})}$ as generating curves, we introduce the Lissajous nodes as the sampling points
\begin{equation} \label{eq:LSpoints}
\bs{\mathrm{LS}}_{\epsilon}^{(\bs{n})} = \left\{\gamma_{\epsilon}^{\bs{n}}\left(\frac{\pi k}{\epsilon n_1n_2}\right), \quad k=0,...,2\epsilon n_1n_2-1 \right\}.
\end{equation}
In our upcoming tests, we will use the points $\bs{\mathrm{LS}}_{2}^{(\bs{n})}$, with $n_1$, $n_2$ relatively prime and $n_1 + n_2$ odd
as underlying interpolation nodes. These node sets were already used in \cite{Marchetti,EKDA2015,KEASKB2016} for applications in MPI. The number of points is given by $\# \bs{\mathrm{LS}}_{2}^{(\bs{n})} = 2 n_1 n_2 + n_1 + n_2$, see \cite{erb,EKDA2015}. The fill distance $$h_{ \bs{ \mathrm{LS}}_{2}^{(\bs{n})}} =  \ds \max_{\bs{y} \in [\text{-}1,1]^2} \min_{  \bs{x} \in \bs{\mathrm{LS}}_{2}^{(\bs{n})}} \|\bs{x} - \bs{y}\|_2$$ for the nodes $\bs{\mathrm{LS}}_{2}^{(\bs{n})}$ in the square $[-1,1]^2$ can be computed as
\begin{equation}
 h_{ \bs{ \mathrm{LS}}_{2}^{(\bs{n})}} =  \frac{1}{2} \max \! \left\{ \!\! \sqrt{ \! S_{n_1}^2\!\! + \!\! \ts \left( \frac{S_{2n_1}^2\!\! + S_{2n_2}^2\!\! - S_{n_1} S_{2n_1}}{S_{2n_2}} \!\! \right)^2}, \! \sqrt{ \! S_{n_2}^2\!\! + \!\! \ts \left( \frac{S_{2n_1}^2\!\! + S_{2n_2}^2\!\! - S_{n_2} S_{2n_2}}{S_{2n_1}} \!\! \right)^2}  \right\},\label{eq:filldistanceLS} \end{equation}
by using the shortcut $S_n = \sin ( \pi / n )$.
This allows us to express the fill distance of $\bs{\mathrm{LS}}_{2}^{(\bs{n})}$ directly
in terms of the frequency parameters $n_1$ and $n_2$. Further, we have the estimates
\[\frac{1}{2} \max\left\{ S_{n_1}, S_{n_2}\right\} \leq h_{\bs{\mathrm{LS}}_{2}^{(\bs{n})}} \leq  \max\left\{ S_{2n_1}, S_{2n_2}\right\} \leq \max\left\{ \frac{\pi}{2n_1}, \frac{\pi}{2n_2}\right\}.\]
For $\bs{n} = (32,33)$ and $\bs{n} = (10,11)$, the nodes $\bs{\mathrm{LS}}_{2}^{(\bs{n} )}$ are illustrated in
Figure \ref{fig:1} (right).

\subsubsection{Shepp-Logan phantom}
As a main test phantom with sharp edges, we use the Shepp-Logan phantom
$f_{\mathrm{SL}}$ on $\Omega = [-1,1]^2$ as introduced in \cite{SheppLogan}. The function $f_{\mathrm{SL}}: [-1,1]^2 \to [0,1]$ is defined as a composition of $10$ step functions determined by
elliptic equations. A discretization of $f_{\mathrm{SL}}$ on an equidistant $M \times M$ grid, $M = 150$, is displayed in Figure \ref{fig:1} (left).
\subsubsection{Geometric phantom} As a second phantom we use a geometric composition $f_{\mathrm{G}}$ of an ellipse $\mathrm{E}$, a rectangle $\mathrm{R}$ and a bounded parabola $\mathrm{P}$, discretized on a $M \times M$ grid of size $M = 150$. The function $f_{\mathrm{G}}$ on $\Omega = [-1,1]^2$ is given as
$f_{\mathrm{G}} = \chi_{\mathrm{E}} + 1.5 \chi_{\mathrm{R}} + 2 \chi_{\mathrm{P}}$, where $\chi_{\mathrm{E}}$, $\chi_{\mathrm{R}}$ and $\chi_{\mathrm{P}}$ denote the characteristic functions of $\mathrm{E}$, $\mathrm{R}$ and $\mathrm{P}$, respectively. The phantom $f_{\mathrm{G}}$ is illustrated in Figure \ref{fig:1} (middle, left).

\subsubsection{Kernels}
For RBF interpolation in $\R^d$, as well as for the VSDK interpolation scheme which requires a kernel in $\R^{d+1}$, we use the following RBFs (cf. \cite{Fasshauer15}):

\begin{enumerate}[topsep=2pt, partopsep=2pt,itemsep=2pt,parsep=2pt,leftmargin=3em,labelwidth=*,align=right,label=(\roman*)]
\item The C$^0$-Mat{\'e}rn function $\phi_{\mathrm{Mat},0}(r) = e^{-r}$. The native space of the corresponding kernel is exactly the Sobolev space $\mathrm{W}_2^s(\R^d)$ with $s = \frac{d+1}{2}$.
We have $\mathrm{W}_2^{\frac{d+1}{2}}(\R^d) \subset C^0(\R^d)$.
\item The C$^2$-Mat{\'e}rn function $\phi_{\mathrm{Mat},2}(r) = (1+r)e^{-r}$. The native space of $\phi_{\mathrm{Mat},2}(\|\bs{x}\|_2)$ is the Sobolev space $\mathrm{W}_2^s(\R^d)$ with $s = \frac{d+3}{2}$. Functions in $\mathrm{W}_2^{\frac{d+3}{2}}(\R^d)$ are contained in $C^2(\R^d)$.
\item The C$^4$-Mat{\'e}rn function $\phi_{\mathrm{Mat},4}(r) = (3+3r+r^2)e^{-r}$. The radial function $\phi_{\mathrm{Mat},4}(\|\bs{x}\|_2)$ generates the Sobolev space $\mathrm{W}_2^s(\R^d)$ with $s = \frac{d+5}{2}$. $\mathrm{W}_2^{\frac{d+5}{2}}(\R^d)$ is contained in $C^4(\R^d)$.
\item The Gauss function $\phi_{\mathrm{Gauss}}(r) = e^{-r^2}$. This is an analytic function. The native space for the Gauss kernel is contained in every Sobolev space $\mathrm{W}_2^s(\R^d)$, $s \geq 0$.
\end{enumerate}

With decreasing separation distance of the interpolation nodes, the calculation of the coefficients in \eqref{eq:coeffVSDK} can be badly conditioned when solving the linear system  directly. This is particularly the case when using the Gaussian as underlying kernel. In order to stabilize the calculation, we regularized the system \eqref{eq:coeffVSDK} by adding a small multiple $\lambda>0$ of the identity to the interpolation matrix (we chose $\lambda = 10^{-12}$ in our calculations). Note however that in the literature there exist more sophisticated ways to avoid this bad conditioning, see for instance \cite[Chapters 11,12 \& 13]{Fasshauer15}.

\subsection{Experiment 1 - Convergence for a priori known discontinuities}

\begin{figure}[htbp]
	\centering
	\minipage{0.5\textwidth}
	\centering
	\includegraphics[width= 1\textwidth]{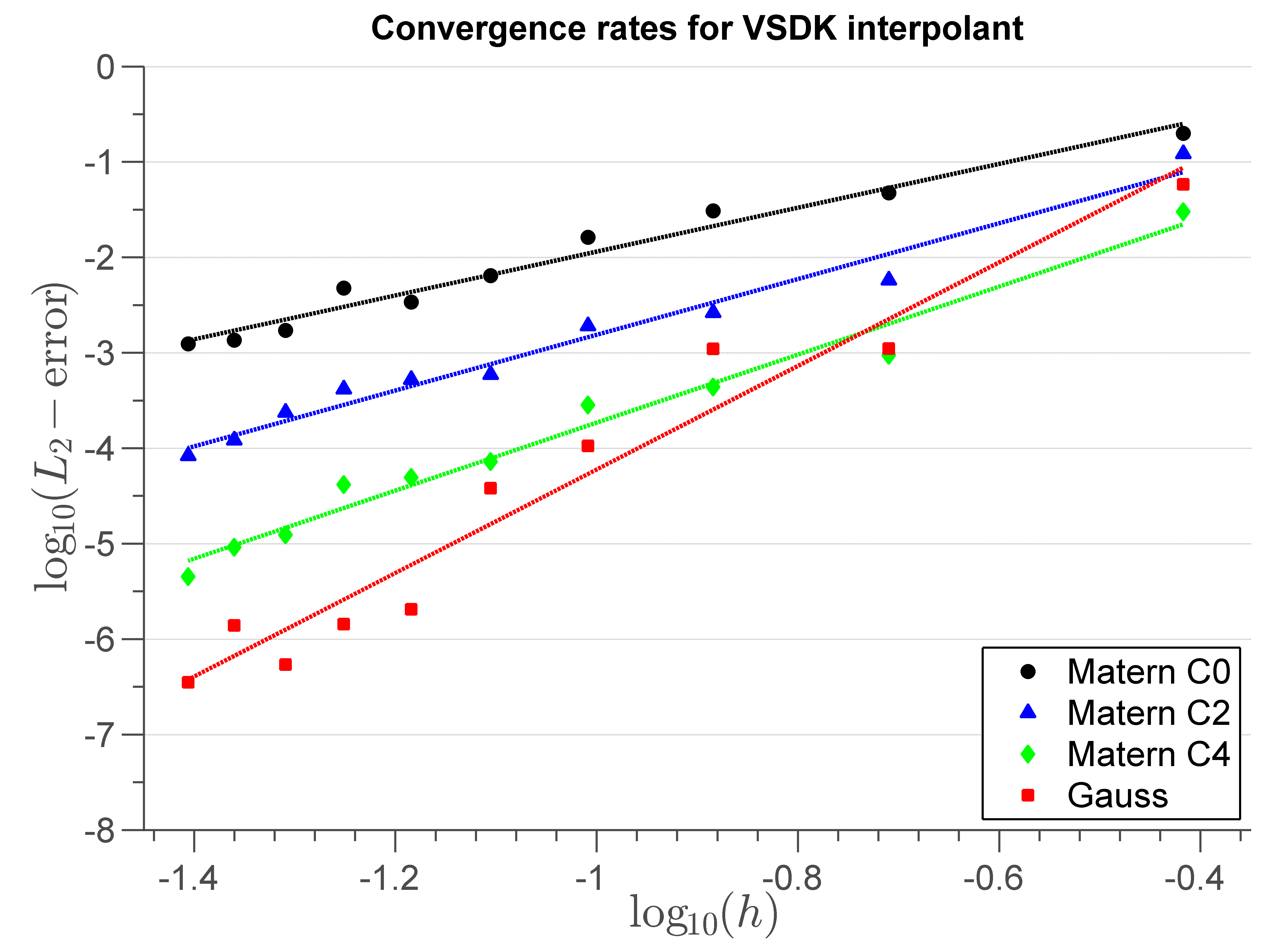}
	\endminipage\hfill
	\minipage{0.5\textwidth}
	\centering
	\includegraphics[width= 1\textwidth]{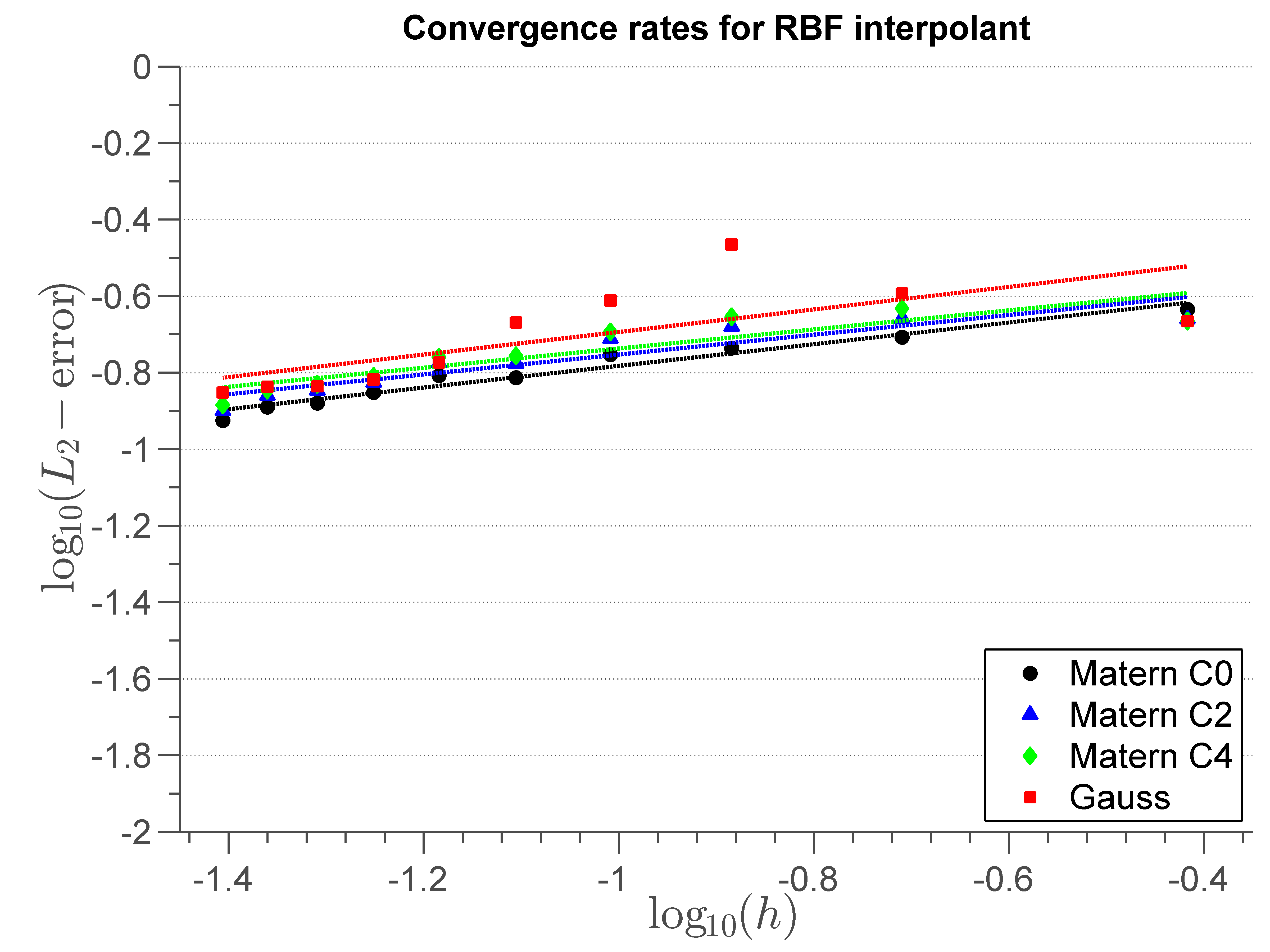}
	\endminipage\hfill
	\caption{Convergence rates for interpolating the Shepp-Logan phantom on the nodes $\bs{\mathrm{LS}}_2^{(\bs{n})}$ using VSDK schemes with a priori known
	scaling function $\psi$ (left) and RBF schemes (right). } \label{fig:2}
\end{figure}

\subsubsection{Description} For $n \in \{4,8,12, \ldots,40\}$ we interpolate the Shepp-Logan phantom
at the Lissajous nodes $\bs{\mathrm{LS}}_2^{(n,n+1)}$. We use the four kernels introduced in the previous Section \ref{sec31} and compute the RBF as well as the VSDK interpolant for all sets of Lissajous nodes. In a $\log$-$\log$ diagram we plot the fill distance $h = h_{ \bs{ \mathrm{LS}}_{2}^{(\bs{n})}}$ given in \eqref{eq:filldistanceLS} against the $L_2$-error between the original Shepp-Logan function $f_{\mathrm{SL}}$ and the interpolant. As an approximation of the continuous $L_2$-error we use the root-mean-square error on the finite
discretization grid. As a scaling function for the VSDK interpolant, we use $\psi(\bs{x}) = 0.5 f_{\mathrm{SL}}(\bs{x})$, i.e. we use a scaling function with the correct a priori information of the discontinuities. The two $\log$-$\log$ diagrams are displayed in Figure \ref{fig:2}. The slopes for the regression lines are listed in Table \ref{tab:1}. The RBF and the VSDK reconstruction
for the interpolation nodes $\bs{\mathrm{LS}}_2^{(40,41)}$ using the $C^2$-Mat{\'e}rn kernel are shown in Figure \ref{fig:3}.

\subsubsection{Results and discussion} This first numerical experiment confirms the theoretical error estimates given in Theorem \ref{thm:errorestimate} and shows that, if the discontinuities of a function are a priori known, the interpolation model based on the discontinuous kernels is significantly better than RBF interpolation.

The slope of the regression line for the RBF interpolation is for all four kernels between $0.25$ and $0.31$. This is in line with the low order smoothness of the Shepp-Logan phantom $f_{\mathrm{LS}}$ which is contained in the Sobolev space $\mathrm{W}_2^{s}([-1,1]^2)$ only for $s < 1/2$. On the other hand, with the choice $\psi = 0.5 f_{\mathrm{LS}}$, the Shepp-Logan phantom is contained in
all piecewise Sobolev spaces $\mathrm{WP}_{2}^s([-1,1]^2)$, $s \geq 0$. Thus,
as predicted in Theorem \ref{thm:errorestimate}, the convergence rates are determined by the smoothness $s$ of the applied kernel. In particular, we obtain the diversified and faster convergence displayed in Figure \ref{fig:2} (left) and Table \ref{tab:1}. Note that for a better comparison between the two interpolation schemes we used for both the global
fill distance $h_{ \bs{ \mathrm{LS}}_{2}^{(\bs{n})}}$ given in \eqref{eq:filldistanceLS}, whereas in Theorem \ref{thm:errorestimate} the fill distance $h$ depends on the segmentation of the domain. In general, we have $h_{ \bs{ \mathrm{LS}}_{2}^{(\bs{n})}} \leq h$.

\begin{table}
\centering
\begin{tabular}{| l | l | l | l |}
  \hline			
  Kernel & Smoothness $s$ & Slope VSDK convergence& Slope RBF convergence \\
  \hline
  $C^0$-Mat{\'e}rn & 1.5 & 2.3609 & 0.29117 \\
  $C^2$-Mat{\'e}rn & 2.5 & 2.9918 & 0.26692 \\
  $C^4$-Mat{\'e}rn & 3.5 & 3.6521 & 0.25791 \\
  Gauss & analytic & 5.5690 & 0.30625 \\
  \hline
\end{tabular}
\caption{The slopes of the convergence rates in Figure \ref{fig:3} for VSDK and RBF interpolation.}
\label{tab:1}
\end{table}

\begin{figure}[htbp]
	\centering
	\minipage{0.5\textwidth}
	\centering
	\includegraphics[width= 1\textwidth]{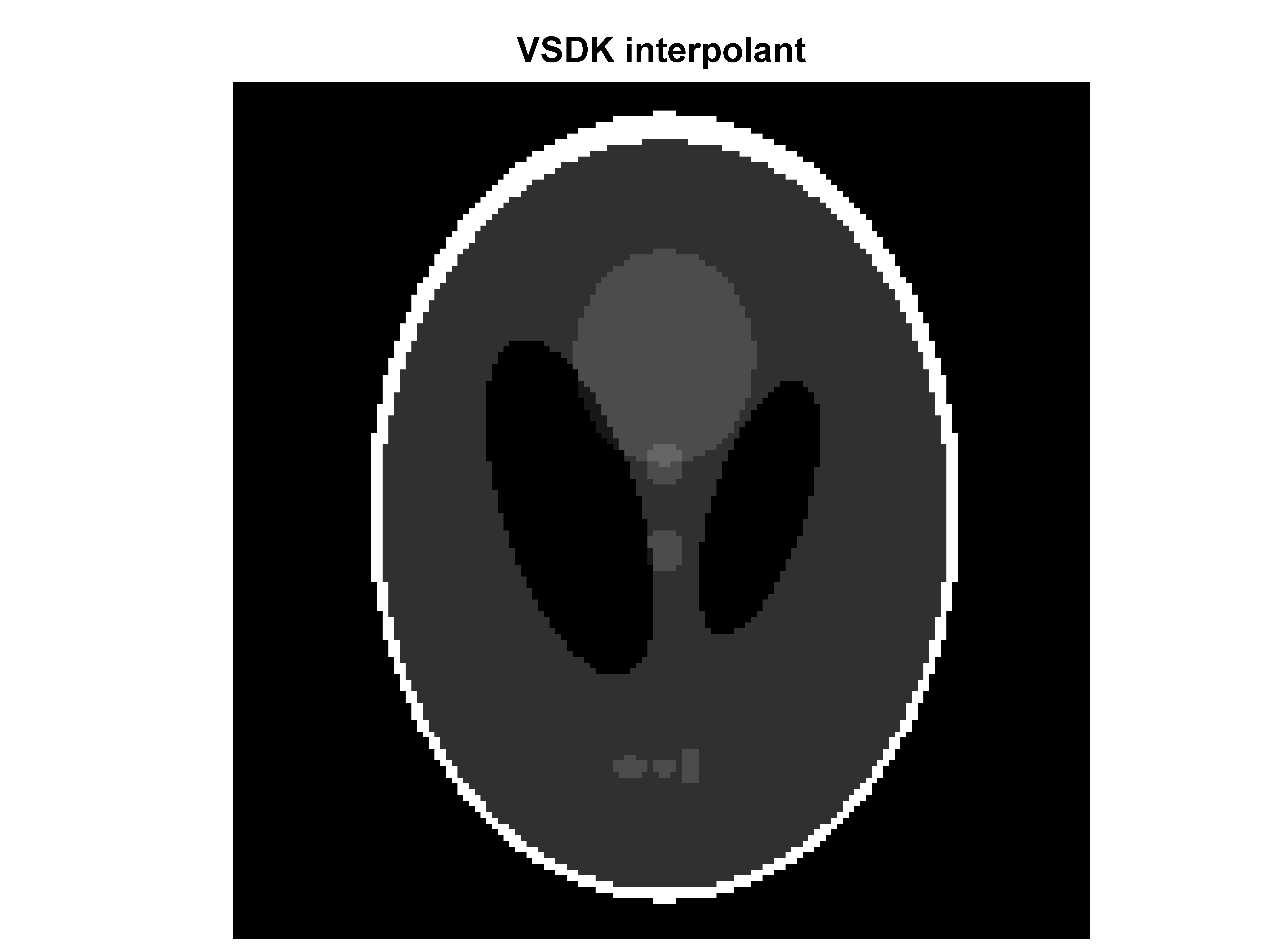}
	\endminipage\hfill
	\minipage{0.5\textwidth}
	\centering
	\includegraphics[width= 1\textwidth]{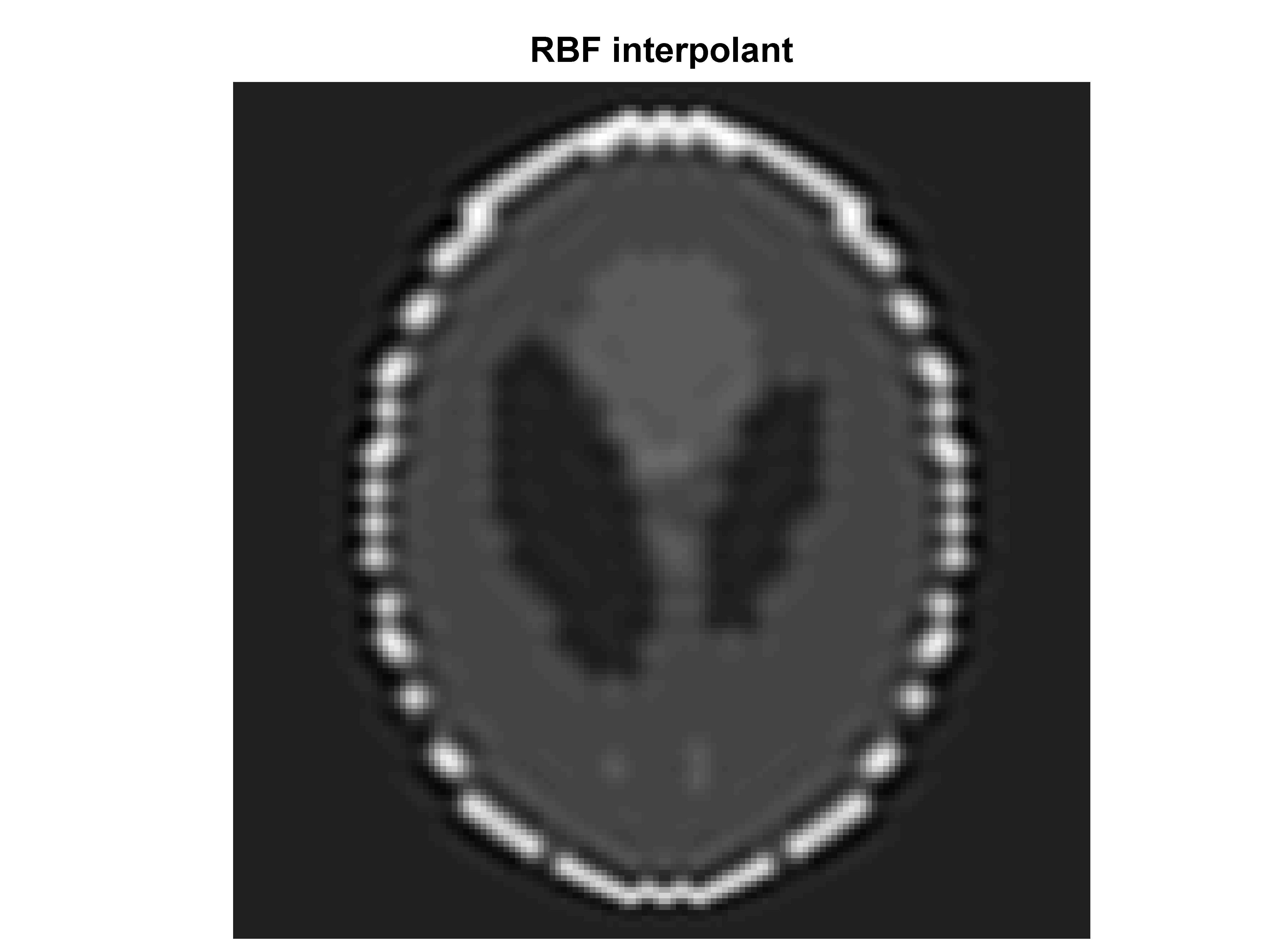}
	\endminipage\hfill
	\caption{Reconstruction of the Shepp-Logan phantom for given sampling data at the nodes $\bs{\mathrm{LS}}_{2}^{(40,41)}$ using the VSDK scheme (left) and the RBF interpolation (right). For both schemes, the C$^2$-Mat{\'ern} kernel is used. In the VSDK scheme, the scaling function $\psi$ is a priori known.} \label{fig:3}
\end{figure}

\begin{figure}[htbp]
	\centering
	\centering
	\includegraphics[width= 1\textwidth]{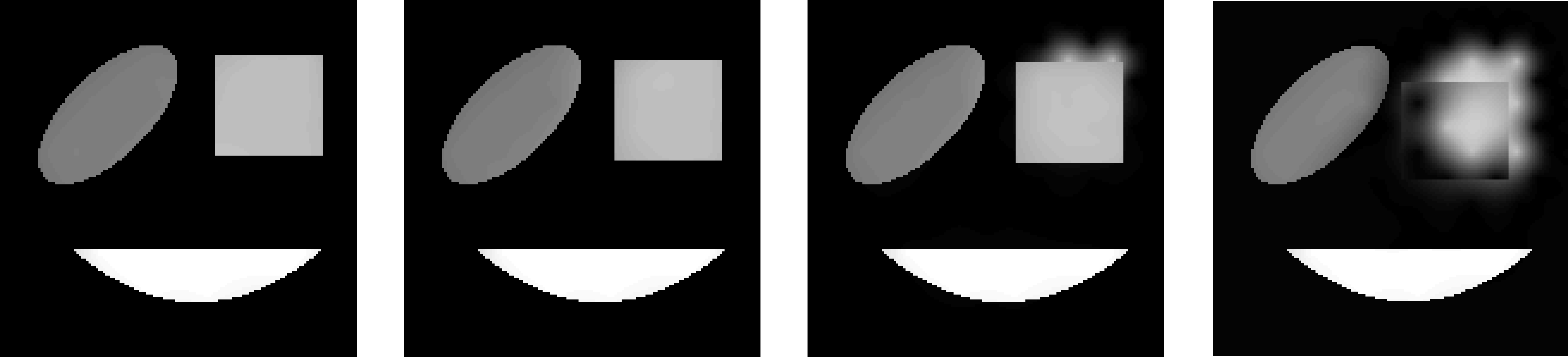}
	\caption{Reconstruction of the geometric phantom by the VSDK scheme with differing scaling functions. The set of interpolation nodes is $\bs{\mathrm{LS}}_{2}^{(10,11)}$. }
	\label{fig:4}
\end{figure}

\begin{figure}[htbp]
	\centering
	\includegraphics[width= 1\textwidth]{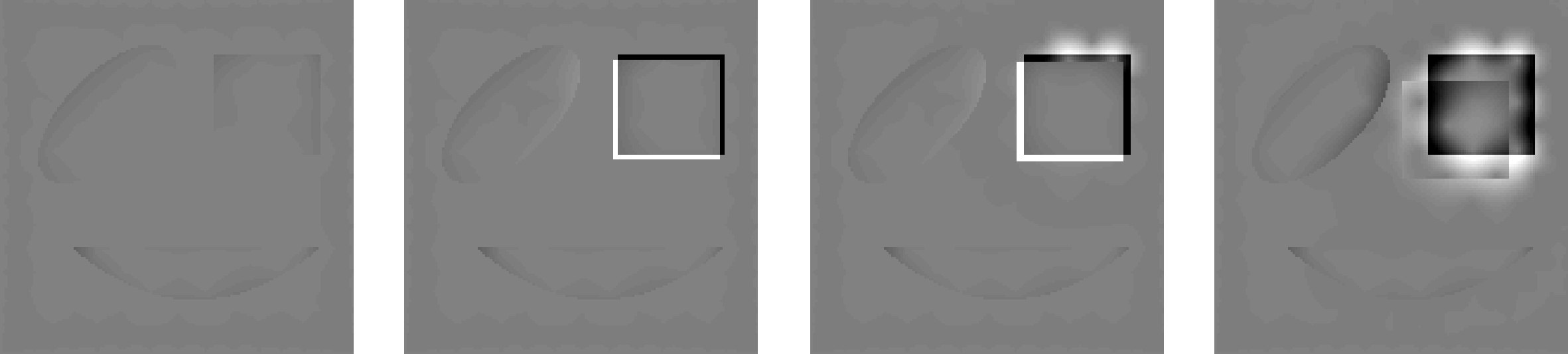}
	\caption{The differences of the VSDK interpolants in Figure \ref{fig:4} to the original phantom.}
	\label{fig:5}
\end{figure}

\newpage

\subsection{Experiment 2 - Perturbations of the scaling function} \label{sec:shiftmask}

\subsubsection{Description} In the second computational experiment, we test the sensitivity of the VSDK interpolation with respect to
shifts of the scaling function $\psi$. For this experiment, we consider the geometric phantom
$f_{\mathrm{G}}$ and interpolate it with differing scaling functions at the Lissajous nodes
$\bs{\mathrm{LS}}_{2}^{(10,11)}$ using the $C^0$-Mat{\'e}rn kernel. The corresponding reconstructions are displayed in Figure \ref{fig:4}. Starting from the correct scaling function (left), i.e. using $\psi = f_{\mathrm{G}}$, the rectangle in the scaling function $\psi$ is slowly shifted towards the center (in Figure \ref{fig:4}, from left to right). The corresponding interpolation errors with respect to the original function $f_{\mathrm{G}}$ are shown in Figure \ref{fig:5}.

\subsubsection{Results and discussion} The outcome of the variably scaled kernel interpolation depends sensitively on the choice of the scaling function $\psi$. All reconstructions in Figure \ref{fig:4} interpolate the function values on the Lissajous nodes. If the scaling function is correctly chosen (left) or only slightly shifted (middle, left), no artifacts are visible in the interpolation. However, the larger the shift of the rectangle in the scaling function $\psi$ gets, the stronger the artifacts are. In particular we see that if the values of $\psi$ do not correspond to the data values on the interpolation nodes, Gibbs type artifacts appear. Therefore, if the VSDK interpolation scheme is applied in a setting in which the edges are not known, a robust edge estimator is needed. In the next section, we will discuss some possibilities for such an estimator.

\section{Extracting edges from the given data} \label{sec4}

We use algorithms from machine learning to obtain a segmentation of the domain $\Omega$.
In particular, we focus on the so-called Support Vector Machines (SVMs) and refer to \cite{Scholkopf,Smola} for a general overview. The main reason to use kernel machines for segmentation is that they can be applied directly to scattered data. Note however that the literature on segmentation of images and edge detection is very extensive and gives a lot of further interesting possibilities to obtain a segmentation. We refer to \cite{SolomonBreckon} for a general introduction.

\subsection{Segmentation of an image by classification algorithms} \label{sec:kernelmachine}

In order to obtain a classification of the entire domain $\Omega$, we separate the data values $(\bs{x}_i,f_i)$ into $n$ classes $\mathcal{S}_1, \ldots, \mathcal{S}_n$ such that all nodes $\bs{x}_i$ in one class $\mathcal{S}_j$
are precisely contained in $\Omega_j$, $j \in \{1, \ldots, n\}$. We link every class $\mathcal{S}_j$ to a value $\alpha_j \in \R$ and set the label $z_i = \psi(\boldsymbol{x}_i) = \alpha_j $ if
$(\bs{x}_i,f_i)$ is contained in $\mathcal{S}_j$. From the labels $\mathcal{Z} = \{z_1, \ldots, z_N\}$ of the points in ${\cal X}$, we want to derive now a classification for every $\bs{x} \in \Omega$.

We give a short description of SVM classification, a more precise introduction can be found, for instance, in \cite{Scholkopf,Smola}. To simplify the considerations, we assume that we only have two classes with possible label values $\alpha_1 = -1$ and $\alpha_2 = 1$.
In this case, a decision function $z$ that allows us to assign to every $\bs{x}$ an appropriate label is given by
$$z(\boldsymbol{x}) = \textrm{sign}(h( \boldsymbol{x})),$$
where $h = 0$ describes a hyperplane separating the given measurements. The hyperplane is set up
with help of the so called kernel trick. By virtue of Mercer's theorem \cite{Mercer}, any kernel $K$ can be decomposed as
	\begin{equation} \label{eq5}
	K(\boldsymbol{x}, \boldsymbol{y}) = \bs{\Theta}(\bs{x})^{\intercal} \bs{\Theta}(\bs{y})
	 =  \sum_{j=1}^{\infty} \Theta_j(\boldsymbol{x}) \Theta_j(\boldsymbol{y}),\quad \boldsymbol{x}, \boldsymbol{y}\in\Omega,
	\end{equation}
	where $\Theta_j$ are eigenfunctions of the integral operator $g \to \int_{\Omega} K(\boldsymbol{x},\boldsymbol{y})  g(\boldsymbol{y})  d\boldsymbol{y}$.
The kernel trick consists in mapping the points $\bs{x}_i$ via $\bs{\Theta}$ into
a (possible) infinite dimensional Hilbert space and to describe the separating hyperplane as
\begin{equation*}
h(\boldsymbol{x})= \bs{\Theta}(\boldsymbol{x})^{\intercal} \boldsymbol{w}+b.
\end{equation*}
The weight $\boldsymbol{w}$, i.e. the unit normal vector to the hyperplane, and the bias $b$ can  be determined by maximizing the gap to both sides of this hyperplane.
One standard way to obtain this hyperplane is by solving the optimization problem
\begin{equation*}
\max_{\bs{\beta}} \left( \sum_{k=1}^{N} \beta_k - \dfrac{1}{2} \sum_{k=1}^{N} \sum_{i=1}^{N} \beta_k \beta_i z_{k} z_{i} K(\boldsymbol{x}_k, \boldsymbol{x}_i) \right),
\end{equation*}
subject to the constraints
\begin{equation*}
\left\{
\begin{array}{l}
\sum_{k=1}^{N} \beta_k z_{k} = 0,\\
0 \leq \beta_i \leq C,  \ i \in \{1, \ldots, N\}.
\end{array}\right.
\end{equation*}
Here, the box constraint $C$ is simply a regularization parameter \cite{Fasshauer15}. Based on the maximizer of this problem, the decision function $z$ of the SVM classifier is then given as
\begin{equation} \label{eq:decisionfunction}
z(\boldsymbol{x})= \textrm{sign}( h(\boldsymbol{x})) = \textrm{sign} \left( \sum_{i=1}^{N} \beta_i z_{i} K(\boldsymbol{x},\boldsymbol{x}_i)+b \right).
\end{equation}
The bias $b$ can be determined as
\begin{equation*}
b = \sum_{k=1}^{N} \beta_k z_{k} K(\boldsymbol{x}_k,\boldsymbol{x}_j),
\end{equation*}
where $j$ denotes the index of a coefficient $\beta_j$ which is strictly between $0$ and $C$.

The classification function $z(\bs{x})$ in \eqref{eq:decisionfunction} gives now the desired segmentation of $\Omega$: the two sets $\Omega_1$ and $\Omega_2$ are defined such that for $\bs{x} \in \Omega_i$ we have $z(\bs{x}) = \alpha_i$, $i \in \{1,2\}$. The corresponding discontinuous scaling function $\psi$ on $\Omega$ is given as
\begin{equation}
\psi(\boldsymbol{x}) = \left\{
\begin{array}{ll}
\alpha_1, & \quad \textrm{if} \hskip 0.2cm z(\boldsymbol{x}) = \alpha_1,\\
\alpha_2, & \quad \textrm{if} \hskip 0.2cm z(\boldsymbol{x}) = \alpha_2,\\
\end{array}\right.
\label{map_v}
\end{equation}
It is straightforward to extend this classification scheme if $\alpha_{1}, \alpha_2 \neq \pm 1$.
There are also several strategies to extend this scheme to the case that the number of classes is $n \geq 2$. In the literature, this is known as Multiclass SVM classification.
One usual approach here is to divide the single multiclass problem into multiple binary
SVM classification problems.

\subsection{Strategies to set the labels for classification} \label{sec:labelstrategies}
In general, the choice of a labeling strategy depends on the aimed at application. In the following, we specify a few simple heuristic strategies to extract the labels $\mathcal{Z}$ from a given data set
$(\mathcal{X},\mathcal{F})$.

\subsubsection{Using thresholds on the data values} \label{sec:manualstrategy1}
If the function $f$ has discontinuities, these are visible as deviations in the data set $\mathcal{F}$. A very simple strategy is therefore to use thresholds for the definition of the labels. If $a_0 < a_1 < \cdots < a_n$ and $\mathrm{supp} (\mathcal{F})$ is contained in the interval $[a_0, a_n)$ we can define $n$ classes $\mathcal{S}_1, \ldots, \mathcal{S}_n$ by
assigning $(\bs{x}_i,f_i)$ to $\mathcal{S}_j$ if $a_{j-1} \leq f_i < a_j$, $j \in \{1, \ldots,n\}$.

\subsubsection{Using thresholds on interpolation coefficients} In \cite{romani}, it is shown that variations in the expansion coefficients of an RBF interpolation can be used to detect the edges of a function. Thus, in the same way as in the previous strategy, thresholds on the absolute value of the RBF coefficients can be applied to determine the aimed at labeling.

\subsubsection{Automated strategies using $\mathrm{k}$-means clustering} \label{sec:autostrategy}
The given data $(\mathcal{X},\mathcal{F})$ can also be segmented using an automated procedure using $\mathrm{k}$-means clustering. If the size $n$ of classes is known, this method provides $n$ pairwise disjoint classes $\mathcal{S}_1, \ldots, \mathcal{S}_n$ by minimizing the functional
\[\sum_{j=1}^n \sum_{(\bs{x},h) \in \mathcal{S}_j} |h - \bar{h}_j|. \]
The value $\bar{h}_j$ denotes the mean of all function values $h$ inside the class $\mathcal{S}_j$. Note that in this case the position $\bs{x}$ of the data is not used to determine the labels.

\subsection{Algorithm for VSDK interpolation with unknown edges}

In the following Algorithm \ref{alg:41}, we summarize the entire scheme for the computation of a shape-driven interpolant from given function values on a node set $\cal{X}$ and unknown discontinuities. For the interpolation, the VSDK scheme introduced in Section \ref{subsec:VSDK} is used. To estimate the edges of $f$, we use the segmentation and labeling procedures described in Section \ref{sec:kernelmachine} and Section \ref{sec:labelstrategies}.

\begin{algorithm}
\caption{Shape-driven interpolation with discontinuous kernels}
\label{alg:41}
 			\begin{tabular}{p{15cm}*{1}{c}}
 				\vskip 0.01 cm
 				{\fontfamily{pcr} \selectfont INPUTS:} Set of interpolation nodes \vskip 0.08 cm
 				\hskip 2.2 cm
 				${\cal X}=\{\boldsymbol{x}_i, \hskip 0.1cm i=1,\ldots,N\} \subseteq \Omega$, 						\vskip 0.08 cm
 				a corresponding set of data values
 				\vskip 0.08 cm
 				\hskip 2.2 cm ${\cal F}=\{f_i = f(\bs{x}_i),\hskip 0.1cm i=1,\ldots,N\}$,
 				\vskip 0.08 cm and the desired evaluation point(s) $\bs{x} \in \Omega$.
 				\vskip 0.12 cm
 				{\fontfamily{pcr} \selectfont OUTPUTS:} VSDK interpolant ${V}_f({\boldsymbol{x}})$, 
 				\vskip 0.24 cm
 				{\fontfamily{pcr} \selectfont Step 1:}
 				Extract the labels ${\cal Z}$ for ${\cal X}$ using ${\cal F}$ with a strategy of Section \ref{sec:labelstrategies}.
 				\vskip 0.12 cm
 				{\fontfamily{pcr} \selectfont Step 2:}
 				Train the kernel machine in Section \ref{sec:kernelmachine} with the points ${\cal X}$ and the labels  \vskip 0.08 cm
 				\hskip 2.0 cm ${\cal Z}$
 				to obtain a prediction \eqref{map_v} for the scaling function $\psi$.
 				\vskip 0.12 cm
 				{\fontfamily{pcr} \selectfont Step 3:} Calculate the coefficients $\boldsymbol{c}_i$, $i \in \{1, \ldots, N\}$, of the VSDK interpolation \vskip 0.08 cm
 				\hskip 2.0 cm
 				 by solving \eqref{eq:coeffVSDK}.
 				\vskip 0.12 cm
 				{\fontfamily{pcr} \selectfont Step 4:}
 				Evaluate the interpolant ${\cal V}_f({\boldsymbol{x}})$ in \eqref{eq3} at $\bs{x} \in \Omega$.
 				\\[\smallskipamount]
 			\end{tabular}
 \end{algorithm}

 \begin{figure}[htbp]
	\centering
	\includegraphics[width= 1\textwidth]{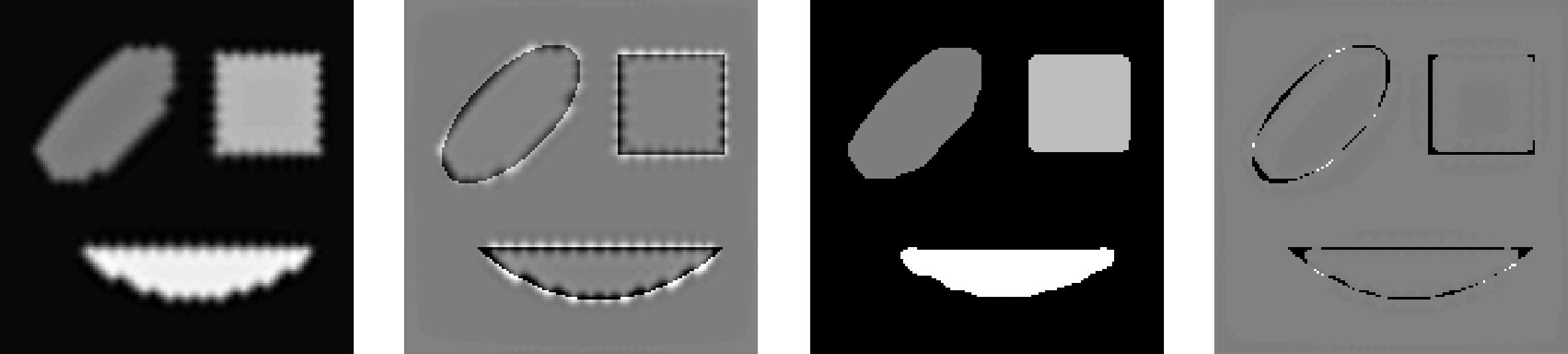}
	\caption{Comparison of RBF interpolation (left) with the VSDK scheme given in Algorithm \ref{alg:41} (right). The interpolation is performed on the nodes $\bs{\mathrm{LS}}_{2}^{(33,32)}$. In the second and the fourth image (from the left) the respective differences to the original phantom are displayed.}
	\label{fig:5b}
\end{figure}

\subsection{Numerical example}
\subsubsection{Description}
On the node set $\bs{\mathrm{LS}}_2^{(33,32)}$, we interpolate the geometric phantom $f_{\mathrm{G}}$ using an ordinary RBF interpolation and the VSDK scheme from Algorithm \ref{alg:41}. In both cases, we use the $C^0$-Mat{\'e}rn function as underlying kernel. The applied edge estimator in Algorithm \ref{alg:41} is based on the segmentation method of Section \ref{sec:kernelmachine} and the automated labeling described in Section \ref{sec:autostrategy}.
The resulting RBF interpolant and the error with respect to the original phantom are displayed in Figure \ref{fig:5b} (left). In Figure \ref{fig:5b} (right) the outcome of Algorithm \ref{alg:41} and the respective error with respect to $f_{\mathrm{G}}$ are shown.

\subsubsection{Results and discussion} In this example in which we don't use the a priori knowledge of the discontinuities, the VSDK scheme in combination with the edge estimator gives a higher reconstruction quality than an ordinary RBF interpolation. In particular, in the VSDK interpolation the Gibbs phenomenon is not visible and the errors are more localized at the boundaries of the geometric figures. For a more quantitative comparison, we compute the relative discrete $L_1$-errors of the two reconstructions. We obtain
\[ \frac{\|f_{\mathrm{G}} - P_{f_{\mathrm{G}}}\|_1}{\| f_{\mathrm{G}}\|_1} \approx 0.1647, \qquad
 \frac{\|f_{\mathrm{G}} - V_{f_{\mathrm{G}}}\|_1}{\| f_{\mathrm{G}}\|_1} \approx 0.1011,\]
i.e., the VSDK interpolant gives a slightly better result with respect to the $L_1$-norm.

Again, we want to point out that in case that the discontinuities are not a priori given the output of the VSDK interpolation strongly depends on the performance of the edge detector. If the edges are detected in a reliable way, also the final VSDK interpolation has a good overall quality. For this compare Figure \ref{fig:5b} also with the reconstruction in Figure \ref{fig:4} (left) in which the scaling function with the correct information of the edges was used. On the other hand, as discussed in Section \ref{sec:shiftmask}, if the scaling function $\psi$ is badly chosen also the final reconstruction is seriously affected by artifacts.

\section{Applications in Magnetic Particle Imaging} \label{sec5}

\begin{figure}[htbp]
	\centering
	\includegraphics[width= 1\textwidth]{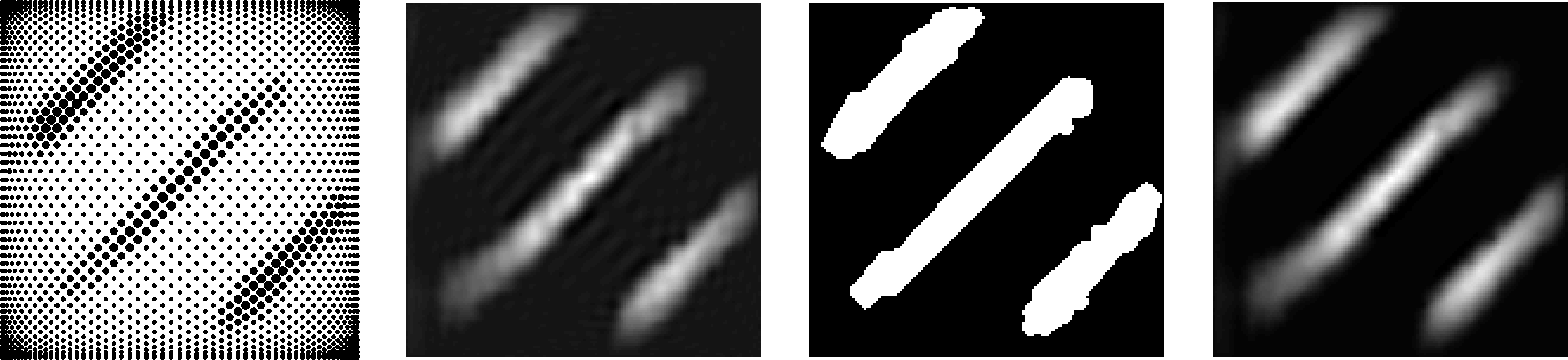}
	\caption{Comparison of different interpolation methods in MPI. The reconstructed data
	on the Lissajous nodes $\bs{\mathrm{LS}}_{2}^{(33,32)}$ (left) is first interpolated
	using the polynomial scheme derived in \cite{erb} (middle left). Using a scaling function constructed     upon a threshold strategy (middle right) the second interpolation is performed by the VSDK scheme (right).}
	\label{fig:6}
\end{figure}

In the early 2000s, B. Gleich and J. Weizenecker \cite{GleichNat}, invented at Philips Research in Hamburg a new quantitative imaging method called Magnetic Particle Imaging (MPI).
In this imaging technology, a tracer consisting of superparamagnetic iron oxide nanoparticles is injected and then detected through the superimposition of different magnetic fields. In common MPI scanners, the acquisition of the signal is performed following a generated field free point (FFP) along a chosen sampling trajectory. The determination of the particle distribution given the measured voltages in the receive coils is an ill-posed inverse problem that can be solved only with proper regularization techniques \cite{knopp}.

Commonly used trajectories in MPI are Lissajous curves \cite{knopp_lissa}. To reduce the amount of calibration measurements, it is shown in \cite{KEASKB2016} that the reconstruction can be restricted to particular sampling points along the Lissajous curves, i.e., the Lissajous nodes $\bs{\mathrm{LS}}_{2}^{(\bs{n})}$ introduced in \eqref{eq:LSpoints}. By using a polynomial interpolation method on the Lissajous nodes \cite{erb} the entire density of the magnetic particles can then be restored. These sampling nodes and the corresponding polynomial interpolation can be seen as an extension of a respective theory on the Padua points \cite{bos,bos1}.

\newpage

If the original particle density has sharp edges, the polynomial reconstruction scheme on the Lissajous nodes is affected by the Gibbs phenomenon. As shown in \cite{Marchetti}, post-processing filters can be used to reduce oscillations for polynomial reconstruction in MPI. In the following, we demonstrate that the usage of the VSDK interpolation method in combination with the presented edge estimator effectively avoids ringing artifacts in MPI and provides reconstructions with sharpened edges.

\subsection{Description}
As a test data set, we consider MPI measurements conducted in \cite{KEASKB2016} on a phantom consisting of three tubes filled with Resovist, a contrast agent consisting of superparamagnetic iron oxide. By the proceeding described in \cite{KEASKB2016} we then obtain a reconstruction of the particle density on the Lissajous nodes $\bs{\mathrm{LS}}_{2}^{(33,32)}$. This reduced reconstruction on the Lissajous nodes is illustrated in Figure \ref{fig:6} (left). A computed polynomial interpolant of this data is shown in Figure \ref{fig:6} (middle, left). In this polynomial interpolant some ringing artifacts are visible. In order to obtain the labeling for the classification algorithm, we use the simple
thresholding strategy described in Section \ref{sec:manualstrategy1} using $1/5$ of the maximal signal strength as a threshold for a binary
classification. The scaling function $\psi$ for the VSDK scheme is then obtained by using the classification algorithm of Section \ref{sec:kernelmachine} with a Gauss function for the kernel machine. The resulting scaling function is visualized in Figure \ref{fig:6} (middle, right). Using the $C^0$-Mat{\'e}rn kernel for the VSDK interpolation, the final interpolant
for the given MPI data is shown in in Figure \ref{fig:6} (right).

\subsection{Results and discussion}
In the polynomial interpolation shown in Figure \ref{fig:6} (middle, left) ringing artifacts are visible. These artifacts could be removed by using Algorithm \ref{alg:41} for the MPI data instead. As an alternative to the applied manual thresholding strategy, it is also possible to use the automated strategy given in Section \ref{sec:autostrategy} in which the $\mathrm{k}$-means algorithms gives the labeling of the data. This second strategy yields classification and reconstruction results that are very close to the ones displayed for the manual strategy in Figure \ref{fig:6}.

\section{Conclusions} \label{sec6}

To reflect discontinuities of a function or an image in the interpolation of scattered data we studied techniques based on the use of variably scaled discontinuous kernels. We obtained a characterization and theoretical Sobolev type error estimates for the native spaces generated by these discontinuous kernels. Numerical experiments confirmed the theoretical convergence rates and investigated the behavior of the interpolants if the scaling function describing the discontinuities is perturbed.

Interpolation with discontinuous kernels can only be conducted if the discontinuities of the function are known. If the discontinuities are not known, sophisticated methods are necessary to approximate the edges from given scattered data. In this work, we used kernel machines, trained with the given data, to obtain the edges and the segmentation of the image.

The results of the VSDK method applied to Magnetic Particle Imaging are promising and show that the Gibbs phenomenon can be sensibly reduced. Work in progress consists in using kernel machines also for regression with VSDKs. This might be of interest when approximating time series with jumps.

\section*{Acknowledgements}

This research has been accomplished within Rete ITaliana di Approssimazione (RITA) and was partially funded by GNCS-IN$\delta$AM and by the European Union's Horizon 2020 research and innovation programme ERA-PLANET, grant agreement no. 689443, via the GEOEssential project.


\begin{thebibliography}{00}

	\bibitem{adams2003}
  	\textsc{R.A. Adams, J. Fournier},
  	\emph{Sobolev Spaces}, Academic Press, London, 2003.
  	
  	\bibitem{bos}
	\textsc{L. Bos, M. Caliari, S. De Marchi, M. Vianello, Y. Xu},
	\emph{Bivariate Lagrange interpolation at the Padua points: the generating curve approach},
	J. Approx. Theory \textbf{143} (2006), pp. 15--25.
	
    \bibitem{bos1}
    \textsc{L. Bos, S. De Marchi, M. Vianello}, \emph{Polynomial approximation
    on Lissajous curves in the $d-$cube}. Appl. Numer. Math. \textbf{116} (2017), pp. 47--56.
	
	\bibitem{Bozzini1}
	\textsc{M. Bozzini, L. Lenarduzzi, M. Rossini, R. Schaback},
	\emph{Interpolation with variably scaled kernels},
	IMA J. Numer. Anal. \textbf{35} (2015), pp. 199--219.
	
	\bibitem{buhmann2000}
	\textsc{M. Buhmann}, \emph{A new class of radial basis functions with compact support},
	Math. Comput. \textbf{70}, 233 (2000), pp. 307-318.
	
	\bibitem{buhmann03}
	\textsc{M. Buhmann}, \emph{Radial Basis Functions: Theory and Implementations}
	(Cambridge Monographs on Applied and Computational Mathematics), Cambridge University Press, Cambridge, 2003
	
	\bibitem{Bushberg}
	\textsc{J.T. Bushberg, J.A. Seibert, E.M. Leidholdt J.M. Boone}, \emph{The essential physics of medical imaging}, 2nd ed. Philadelphia, Pa: Lippincott Williams \&
	Wilkins, 2001.
	
	\bibitem{Canny}
	\textsc{J.F. Canny},  \emph{A computational approach to edge detection},
	IEEE TPAMI \textbf{8} (1986), pp. 34--43.
	
	\bibitem{Czervionke}
	\textsc{L.F. Czervionke, J.M. Czervionke, D.L. Daniels, V.M. Haughton},  \emph{Characteristic features of MR truncation artifacts},
	AJR Am. J. Roentgenol. \textbf{151} (1988), pp. 1219--1228.
	
	\bibitem{Marchetti}
	\textsc{S. De Marchi, W. Erb, F. Marchetti}, \emph{Spectral filtering for
	the reduction of the Gibbs phenomenon for polynomial approximation methods on
	Lissajous curves with applications in MPI},
	Dolomites Res. Notes Approx. \textbf{10} (2017), pp. 128--137.
	
	\bibitem{demarchi18}
	\textsc{S. De Marchi, F. Marchetti, E. Perracchione},
	\emph{Jumping with Variably Scaled Discontinuous Kernels (VSDKs)}, submitted, 2018.
	
	\bibitem{erb}
	\textsc{W. Erb, C. Kaethner, M. Ahlborg, T.M. Buzug},
	\emph{Bivariate Lagrange interpolation	at the node points of
	non-degenerate Lissajous nodes}, Numer. Math. \textbf{133}, 1 (2016), pp. 685--705.	
	
	\bibitem{Erb2016}
	\textsc {W. Erb},
	\emph{Bivariate Lagrange interpolation at the node points of Lissajous
  	curves - the degenerate case},
	Appl. Math. Comput. \textbf{289} (2016), 409--425.
	
	\bibitem{EKDA2015}
    \textsc {W. Erb, C. Kaethner, P. Dencker, M. Ahlborg}
    \emph{A survey on bivariate Lagrange interpolation on Lissajous nodes},
    Dolomites Research Notes on Approximation \textbf{8} (Special issue) (2015), 23--36.


	\bibitem{fuselier}
	\textsc {E. Fuselier, G. Wright},
	\emph{Scattered Data Interpolation on Embedded Submanifolds with Restricted Positive Definite 	Kernels: Sobolev Error Estimates},
	SIAM J. Numer. Anal. \textbf{50}, 3 (2012), pp 1753--1776.
	
    \bibitem{Fasshauer15}
	\textsc{G.E. Fasshauer, M.J. McCourt},
	\emph{Kernel-based Approximation Methods Using} \textsc{Matlab},
	World Scientific, Singapore, 2015.

	\bibitem{Fornberg_gibbs}
	\textsc{B. Fornberg, N. Flyer}, \emph{The Gibbs phenomenon for radial
	basis functions}, in \emph{The Gibbs Phenomenon in Various Representations
	and Applications}. Sampling Publishing, Potsdam, NY, 2008.
	
	\bibitem{GleichNat}
	\textsc{B. Gleich, J. Weizenecker}, \emph{Tomographic imaging using the nonlinear response of 	magnetic particles}, Nature \textbf{435} (2005), pp.  1214--1217.
	
	\bibitem{Gottlieb}
	\textsc{D. Gottlieb, C.W. Shu}, \emph{On the Gibbs phenomenon and its resolution},
	SIAM Review \textbf{39} (1997), pp. 644--668.
	
	\bibitem{Jerri}
	\textsc{A. Jerri}, \emph{The Gibbs Phenomenon in Fourier Analysis, Splines and Wavelet Approximations}, Kluwer Academic Publishers, Dordrecht, Boston, London (1998).	

\newpage

	\bibitem{Jung}
	\textsc{J.-H. Jung}, \emph{A note on the Gibbs phenomenon with multiquadric radial
	basis functions}, Appl. Num. Math. \textbf{57} (2007), pp. 213--219.
	
	\bibitem{JungGottlieb}
	\textsc{J.-H. Jung, S. Gottlieb, S.O. Kim}, \emph{Iterative adaptive RBF methods for detection of edges in two-dimensional
functions}, Appl. Num. Math. \textbf{61} (2011), pp. 77--91.

    \bibitem{KEASKB2016}
    \textsc {C. Kaethner, W. Erb, M. Ahlborg, P. Szwargulski, T. Knopp, T.M. Buzug}
    \emph{Non-Equispaced System Matrix Acquisition for Magnetic Particle Imaging based on Lissajous Node Points},
    IEEE Trans. Med. Imag. \textbf{35}, 11 (2016), pp. 2476--2485.
	
	\bibitem{knopp}
	\textsc{T. Knopp, T.M. Buzug}, \textit{Magnetic Particle Imaging}, Springer-Verlag, Berlin, 		2012.
	
	\bibitem{knopp_lissa}
	\textsc{T. Knopp, S. Biederer, T. Sattel, J. Weizenecker, B. Gleich, J. Borgert, T.M. Buzug}, 	\emph{Trajectory analysis for magnetic particle imaging},
	Phys. Med. Biol. \textbf{54} (2009), pp. 385--397.
	
	\bibitem{Lehmann}
	\textsc{T. M. Lehmann, C. Gonner, K. Spitzer}, 	
	\emph{Survey: interpolation methods in medical image processing},
	IEEE Trans. Med. Imag. \textbf{18}, 11 (1999), pp. 1049--1075.
	
	\bibitem{Mercer}
	\textsc{J. Mercer},
	\emph{Functions of positive and negative type
	and their connection with the theory of integral equations},
	Phil. Trans. Royal Society \textbf{209} (1909), pp. 415--446.
	
	\bibitem{Narcovich05}
	\textsc{F.J. Narcowich, J. D. Ward, H. Wendland},
	\emph{Sobolev bounds on functions with scattered zeros,
	with applications to radial basis function surface fitting.}
	Math. Comput. \textbf{74}, 250 (2005), pp. 743-763.
	
	\bibitem{rieger}
	\textsc{C. Rieger},
	\emph{Sampling inequalities and applications.}
	Disseration, G\"ottingen, 2008.
	
	\bibitem{romani}	
	\textsc{L. Romani, M. Rossini, D. Schenone},  \emph{Edge detection
	methods based on RBF interpolation}, J. Comput. Appl. Math. \textbf{349} (2018), pp. 532--547.
	
	\bibitem{rossini}
	\textsc{M. Rossini}, \emph{Interpolating functions with gradient discontinuities
	via variably scaled kernels}, Dolom. Res. Notes Approx. \textbf{11} (2018), pp. 3--14.
	
	\bibitem{sarra_post}
	\textsc{S.A. Sarra}, \emph{Digital total variation filtering as postprocessing for radial
	basis function approximation methods}, Comput. Math. Appl. \textbf{52}
	(2006), pp. 1119--1130.
	
	\bibitem{Schaback1999}
	\textsc{R. Schaback}
	\emph{Native Hilbert spaces for radial basis functions. I.}
	In: New developments in approximation
	theory (Dortmund, 1998), volume 132 of Internat. Ser. Numer. Math.,
	Birkh\"auser, Basel, 1999, pp. 255-282

	\bibitem{Scholkopf}
	\textsc{B. Sch\"olkopf, A.J. Smola}, \emph{Learning with Kernels: Support Vector Machines,
			Regularization, Optimization, and Beyond}, MIT Press, Cambridge, MA, USA, 2002.
	
	\bibitem{Sharifi}
	\textsc{M. Sharifi, M. Fathy, M.T. Mahmoudi},
	\emph{A Classified and Comparative Study of 	Edge Detection Algorithms},
	in Proc. Int. Conf. on Inform. Technology:
	Coding and Computing, Las Vegas, USA, 2002, pp. 117--120.
	
	\bibitem{SheppLogan}
	\textsc{L.A. Shepp, B.F. Logan}, \emph{The Fourier Reconstruction of a Head Section},
	IEEE Trans. Nucl. Sci. \textbf{NS-21}, 3 (1974), pp. 21--43.
	
	\bibitem{Smola}
	\textsc{A.J. Smola, B. Sch\"olkopf}, \emph{A tutorial on support vector regression},
	Statistics and Computing. \textbf{14} (2004), pp. 199--222.
	
	\bibitem{SolomonBreckon}
	{\sc C.J. Solomon, T.P. Breckon}
	\newblock {\em {Fundamentals of Digital Image Processing: A Practical Approach with Examples in Matlab}}.
	\newblock Wiley-Blackwell, 2010.
	
	\bibitem{Takaki}
	\textsc{A. Takaki, T. Soma, A. Kojima, K. Asao, S. Kamada, M. Matsumoto, K. Murase}, \emph{Improvement of image quality using interpolated projection data estimation method in SPECT},
	Ann. Nucl. Med. \textbf{23}, 7 (2009), pp. 617-626.
	
	\bibitem{Thevenaz}
	\textsc{P. Thevenaz, T. Blu, M. Unser}, \emph{Interpolation revisited}, IEEE Trans. Med. Imag. \textbf{19}, 7 (2000), pp. 739--758.
	
	\bibitem{Wendland05}
	\textsc{H. Wendland},
	\emph{Scattered Data Approximation} (Cambridge Monographs on Applied and Computational Mathematics), Cambridge University Press, Cambridge, 2005.
	
	\bibitem{Zygmund}
	{\sc A. Zygmund}
	\newblock {\em {Trigonometric series, third edition, Volume I \& II combined
  	(Cambridge Mathematical Library)}}.
	\newblock Cambridge University Press, 2002.
		
\end{thebibliography}
\end{document}